\documentclass{amsart}

\usepackage{color, amssymb}

\newtheorem{lemma}{Lemma}[section]
\newtheorem{theorem}[lemma]{Theorem}
\newtheorem{proposition}[lemma]{Proposition}
\newtheorem{corollary}[lemma]{Corollary}

\theoremstyle{definition}

\theoremstyle{remark}

\def\SN{\mathbb N}    
\def\SQ{\mathbb Q}    
\def\SZ{\mathbb Z}    
\def\SP{\mathbb P}    
\def\SR{\mathbb R}    

\def\aa{\mathfrak a}    

\def\int{\mbox{\rm{Int}}}             

\def\Cc{\mathcal C}         
\def\Ii{\mathcal I}
\def\Pp{\mathcal P}         
\def\Oo{\mathcal O}         

\def\e{\varepsilon}

\def\phi{\varphi}

\def\Gal{\mathrm{Gal}}

\def\Ker{\mathrm{Ker}}

\def\be{\begin{equation}}
\def\ee{\end{equation}}

\def\Ima{\mathrm{Im}}
\def\Coker{\textrm{Coker}}
\def\Ker{\mathrm{Ker}}

\begin{document}
\title[Kubota's Satz 4]{An elementary formulation of Kubota's proof\\ of Satz 4 about biquadratic number fields}
\author{Jacques Boulanger and Jean-Luc Chabert}

\address{LAMFA CNRS-UMR 7352,
Universit\'e de Picardie,
80039 Amiens, France}
\email{boulanger.jac@gmail.com}
\email{jean-luc.chabert@u-picardie.fr}


\begin{abstract}
The aim of this note is to give an elementary formulation of Tomio Kubota's proof of Satz 4 in~\cite{bib:kubota1956} since this proposition is a cornerstone for the computation of the order of the P\'olya group of the biquadratic number fields.
\end{abstract}

\maketitle

\section{Introduction and Notations}

\noindent{\bf Standard Notations.}
{\it If $k$ is a number field, we denote by}

\noindent $\bullet$ $\Oo_k$ its ring of integers,

\noindent $\bullet$ $\Oo_k^\times$ its unit group,

\noindent $\bullet$ $\Ii_{k}$ its group of nonzero fractional ideals,

\noindent $\bullet$ $\Pp_{k}$ its group of nonzero fractional principal ideals,

\noindent $\bullet$  $\Cc l(k)=\Ii_k/\Pp_k$ its class group,

\noindent $\bullet$  $\Pi_q(k)$ the product of all the prime ideals of $k$ with norm $q$ if there are some and $\Oo_k$ else, 

\noindent $\bullet$ $\Pp o(k)$ the P\'olya group of $k$, that is, the subgroup of $\Cc l(k)$ 
generated by the classes $\overline{\Pi_q(k)},$

\smallskip

\noindent{\bf General Notations.} {\it If $k$ is a Galois number field, we denote by}

\smallskip

\noindent $\bullet$ $e_p(k/\SQ)$ the ramification index of the prime number $p$ in the extension $k/\SQ$,

\noindent $\bullet$ $G_k$ the Galois group $\Gal(k/\SQ)$,

\noindent $\bullet$ $\Ii_{k}^{G_k}$ the subgroup $\langle\, \Pi_q \mid q\in\SN^* \,\rangle$ of $\Ii_k$ whose elements are stable by $G_k,$

\smallskip

If $k/\SQ$ is Galois, then $\Pi_q(k)$ is principal except possibly when $q$ is a power of a ramified prime and 
$\Pp o(k)$ is the group of the strongly ambiguous classes, that is, the group formed by the classes of the ambiguous ideals:$$\Pp o(k)=\Ii_k^{G_k}/\left(\Pp_k\cap\Ii_k^{G_k}\right)=\Ii_k^{G_k}/\Pp_k^{G_k}.$$
\noindent {\bf Specific Notations.} {\it If $K$ is a biquadratic bicyclic number field, we denote by}

\noindent $\bullet$ $k_i=\SQ(\sqrt{d_i})$ ($i=1,2,3$) its quadratic subfields,

\noindent $\bullet$ $G=\Gal(K/\SQ)=\{id_K,\sigma_1,\sigma_2,\sigma_3) \textrm{ where }\sigma_{i\vert k_i}=id_{k_i},$

\noindent $\bullet$ $q_K=(\Oo_K^\times:\Oo_{k_1}^\times\Oo_{k_2}^\times\Oo_{k_3}^\times)$ its unit index,

\noindent $\bullet$ $\e_i$ the fundamental unit $>1$ of $k_i$ when $k_i$ is real,

\noindent $\bullet$ $\nu_i=1$ or $0$ according to the fact that $k_i$ is real and $N_{k_i/\SQ}(\e_i)=1$ or not,

\noindent $\bullet$ $\nu_K$ the sum $\sum_{i=1}^3\nu_i$,

\noindent $\bullet$ $s_K$ (resp., $s_i$) the number of ramified primes in $K$ (resp., $k_i$),

\noindent $\bullet$ $i_2=1$ or 0 according to the fact that $e_2(K/\SQ)=4$ or not. 

\noindent $\,\bullet$ $j_2=1$ or $0$ according to the fact that $e_2(K/\SQ)=4$ and $\Pi_2(K)$ is not principal or not.

\smallskip

\noindent {\bf Kubota's Technical Notations.}

\noindent $\bullet$ $\Oo_{k_i}^*=\{\eta\in\Oo_{k_i}^\times\mid N_{k_i/\SQ}(\eta)=+1\}$. If $k_i\not\subset\SR$, then $\Oo_{k_i}^*=\Oo_{k_i}^\times$

\noindent $\bullet$ $\Oo_K^*=\left\{\begin{array}{ll}\{\eta\in\Oo_{k_1}^\times\Oo_{k_2}^\times\Oo_{k_3}^\times\mid \eta \textrm{ totally}>0 \textrm{ or} <0\}&\textrm{if }K\subset\SR\\ \;\;\Oo_{k_1}^*\Oo_{k_2}^*\Oo_{k_3}^*&\textrm{if }K\not\subset\SR\end{array}\right.$

\medskip

Clearly, we have:
$\Gal(k_i/\SQ)=\{id_{k_i},\sigma_{j \vert k_i}\}$ and $N_{k_i/\SQ}(\e_i)=\e_i\cdot \sigma_j(\e_i)$ for $j\not=i.$
Moreover, $s_1+s_2+s_3=2s_K+i_2$ since a prime ramified in $K$ is ramified in exactly 2 quadratic subfields except 2 which is ramified in the 3 subfields when $i_2=1$. 
\smallskip

If $k\subset K$, from the extension morphism
$j_{k}^K:\aa\in \Ii_{k} \mapsto \aa\Oo_K\in \Ii_K,$
we deduce the following morphism of groups
$$\psi_1:(\aa_1,\aa_2,\aa_3)\in \Ii_{k_1}\times\Ii_{k_2}\times\Ii_{k_3} \mapsto \aa_1\aa_2\aa_3\Oo_K\in \Ii_K ,$$ 
which, by restriction, leads to another morphism
$$\psi_0:(\aa_1,\aa_2,\aa_3)\in \Ii_{k_1}^G\times\Ii_{k_2}^G \times\Ii_{k_3}^G \mapsto \aa_1\aa_2\aa_3\Oo_K\in \Ii_K^G,$$ 
which itself induces a natural morphism:
$$\psi:\Pp o(k_1)\times \Pp o(k_2)\times \Pp o(k_3)\to \Pp o(K).$$

To compute the cardinality of the P\'olya group $\Pp o(K)$ of the biquadratic number field $K$, we use the classical isomorphism
$\textrm{Im}\,\psi\simeq (\Pp o(k_1)\times \Pp o(k_2)\times \Pp o(k_3))\,/\,\Ker\,\psi.$
Thus, 
$$\vert \Pp o(K)\vert =\vert \Ima\;\psi\vert\times\vert \Coker\;\psi\vert = \vert\Pp o(k_1)\times \Pp o(k_2)\times \Pp o(k_3)\vert \times\frac{\vert \Coker\;\psi\vert }{\vert \Ker\;\psi\vert}$$
Hilbert~\cite[Theorems 105 and 106]{bib:hilbert}, gave the cardinality of the groups $\Pp o(k_i):$
\be\label{eq:I} \vert\Pp o(k_i)\vert = 2^{s_i-1-\nu_i}.\ee 
By~\cite[Proposition 3.1]{bib:chabert2023}, we know also the cardinality of the group $\Coker\;\psi:$
\be\label{eq:II} \Coker\;\psi=\Pp o(K)/\Ima\;\psi\simeq \langle\,\overline{\Pi_2(K)}/\overline{\Pi_2(K)}^2\,\rangle.\ee 
Therefore, $\vert \Coker\;\psi\vert = 2^j$, that is, 2 or 1 according to the fact that 2 is totally ramified in $K$ and $\Pi_2(K)$ is not principal or not.
Thus, it remains to compute the cardinality of $\Ker\,\psi$. Kubota gave this cardinality.

\begin{theorem}\cite[Kubota, Satz 4]{bib:kubota1956} or  \cite[Lemmermeyer, p. 252, line 6]{bib:lemmermeyer1994}.

\noindent The kernel of $\psi:\Pp o(k_1)\times \Pp o(k_2)\times \Pp o(k_3)\to \Pp o(K)$ has the following order:
\be\label{eq:III} \vert \Ker\,\psi\vert =\left\{\begin{array}{cl}\frac{1}{q_K}\prod_pe_p(K/\SQ) & \textrm{if } K \textrm{ is real and } \nu_1=\nu_2=\nu_3=0 \\ \frac{1}{2q_K}\prod_pe_p(K/\SQ) & \textrm{else}. \end{array}\right.\ee
\end{theorem}

Recall that $\left\{\begin{array}{l} q_K\,\vert\, 4\textrm{ if }K\subset \SR\\ q_K\,\vert\, 2 \textrm{ if } K\not\subset\SR
\end{array}\right\}$.
Formulas (\ref{eq:I}), (\ref{eq:II}), and (\ref{eq:III}) lead to:

\begin{corollary}\cite[Theorem 3.3]{bib:chabert2023}
\be \label{eq:55b} \vert \Pp o(K)\vert = \left\{\begin{array}{ll}q_K\times 2^{s_K+j_2-2-\max(1,\nu_K)}& \textrm{if } K \textrm{ is real}\\ q_K\times 2^{s_K+j_2-2-\nu_K}& \textrm{if } K \textrm{ is imaginary} \end{array}\right.\ee 

\end{corollary}

In fact, the most difficult part is Kubota's formula whose proof is difficult to read and our aim here is to give an easier formulation of Kubota's proof itself.

\section{Preliminary lemmas}

The aim of these lemmas is to extend to biquadratic fields the following result about quadratic fields.

\begin{lemma}[Kubota's Hilfssatz 1]\label{lem:ein} Let $k$ be a quadratic number field and $\eta\in\Oo_k^\times$. One can write $\eta=\frac{\rho^2}{r}$ where $\rho\in\Oo_k$ and $r\in\SQ$ if and only if $N_{k/\SQ}(\eta)=+1$.  
\end{lemma}

\begin{proof}
The condition is necessary since $\eta=\frac{\rho^2}{r}$ implies $N_{k/\SQ}(\eta)=\frac{(N_{k/\SQ}(\rho))^2}{r^2}>0$. Conversely, assuming that $N_{k/\SQ}(\eta)=+1$, Hilbert 90 says that there exists $\rho\in \Oo_k$ such that $\eta=\frac{\rho}{\sigma(\rho)}=\frac{\rho^2}{N_{k/\SQ}(\rho)}=\frac{\rho^2}{r}$ where $\sigma$ is the non-trivial automorphism of $k$ and $r\in\SQ$.
\end{proof}

\begin{lemma}\label{lemA}
Assume that $K\subset\SR$ and let  $\eta=sgn(\eta)\, \e_1^{m_1}\e_2^{m_2}\e_3^{m_3}\in\Oo_{k_1}^\times\Oo_{k_2}^\times\Oo_{k_3}^\times$ where $sgn(x)$ denotes the sign of $x$ and $m_1,m_2,m_3\in\SZ$. Then, letting $\lambda_i=N_{k_i/\SQ}(\e_i)$, one has
$$\eta\in\Oo_K^*\;\Leftrightarrow\;\lambda_1^{m_1}=\lambda_2^{m_2}=\lambda_3^{m_3}.$$	
\end{lemma}

\begin{proof}
Since, for $i\not= j$, $\sigma_j(\e_i)=\frac{\lambda_i}{\e_i}$, we have
$$\sigma_1(\eta)=sgn(\eta) \e_1^{m_1}\lambda_2^{m_2}\e_2^{-m_2}\lambda_3^{m_3}\e_3^{-m_3}, \textrm{ and hence, } sgn(\sigma_1(\eta))=sgn(\eta)\,\lambda_2^{m_2}\lambda_3^{m_3}.$$
Thus, $sgn(\sigma_1(\eta))=sgn(\eta)\;\Leftrightarrow \; \lambda_2^{m_2}\lambda_3^{m_3}=1\;\Leftrightarrow \; \lambda_2^{m_2}=\lambda_3^{m_3}.$ By considering one the two other analogous equivalences, we can conclude.
\end{proof}

 \begin{lemma}[Kubota's Hilfssatz 2]$\empty$\label{HS2}
 \noindent $\bullet$ If $K$ is real and if $\lambda_1=\lambda_2=\lambda_3=-1$, then $$\Oo_K^*\simeq \Oo_{k_1}^*\Oo_{k_2}^* \Oo_{k_3}^* \times \SZ/2\SZ.$$ 
 \noindent $\bullet$ Else, $\Oo_K^*= \Oo_{k_1}^*\Oo_{k_2}^* \Oo_{k_3}^*.$
 Moreover, 
\be \label{eq:A} (\Oo_{k_1}^\times\Oo_{k_2}^\times\Oo_{k_3}^\times : \Oo_K^*)=\prod_{i=1}^3(\Oo_{k_i}^\times : \Oo_{k_i}^*)\left\{\begin{array}{ll}\times \frac{1}{2} &\textrm{if } K\subset \SR \textrm{ and } \lambda_i=-1\,(\forall\,i)\\ \times 1 & else\end{array}\right.\ee 
 \end{lemma}

\begin{proof}
$\bullet$ Assume that $K\subset\SR$ and let  $\eta=sgn(\eta)\, \e_1^{m_1}\e_2^{m_2}\e_3^{m_3}\in\Oo_{k_1}^\times\Oo_{k_2}^\times\Oo_{k_3}^\times$.

$\circ$ If $\lambda_1=\lambda_2=\lambda_3=-1,$ by Lemma~\ref{lemA}, $\eta\in\Oo_K^*\;\Leftrightarrow\;m_1\equiv m_2\equiv m_3\pmod{2}\;$

\smallskip

$\Leftrightarrow\;$ The $m_i$'s are either all even or all odd

 $\Leftrightarrow\;\eta=\eta_1\eta_2\eta_3(\e_1\e_2\e_3)^m$
  where $\eta_i\in \Oo_{k_i}^*$  and $m=0$ or $1$ 
 
 \smallskip
 
 $\Leftrightarrow\;\Oo_K^*\simeq \Oo_{k_1}^*\Oo_{k_2}^* \Oo_{k_3}^* \times \SZ/2\SZ$ (since $(\e_1\e_2\e_3)^2\in  \Oo_{k_1}^*\Oo_{k_2}^*\Oo_{k_3}^*$).

\smallskip

$\circ$ Else, for instance,  $\lambda_1=+1$.  Then, by Lemma~\ref{lem:ein},
 $$\eta\in\Oo_K^*\;\Leftrightarrow\;1=\lambda_2^{m}=\lambda_3^{m}\;\Leftrightarrow\;\eta_i\in\Oo_{k_i}^*\;(i=1,2,3)\Leftrightarrow\;\Oo_K^*= \Oo_{k_1}^*\Oo_{k_2}^* \Oo_{k_3}^*.$$
 
 \noindent $\bullet$ Finally, if $K\not\subset \SR$, by definition,  $\Oo_K^*= \Oo_{k_1}^*\Oo_{k_2}^* \Oo_{k_3}^*.$

\medskip

Let us prove now Formula~(\ref{eq:A}).
 If $K$ is real and  $\lambda_1=\lambda_2=\lambda_3=-1$, then $$(\Oo_{k_1}^\times\Oo_{k_2}^\times\Oo_{k_3}^\times : \Oo_K^*)=\frac{1}{2}(\Oo_{k_1}^\times\Oo_{k_2}^\times\Oo_{k_3}^\times : \Oo_{k_1}^*\Oo_{k_2}^* \Oo_{k_3}^*). $$
Else, $$(\Oo_{k_1}^\times\Oo_{k_2}^\times\Oo_{k_3}^\times : \Oo_K^*)=(\Oo_{k_1}^\times\Oo_{k_2}^\times\Oo_{k_3}^\times : \Oo_{k_1}^*\Oo_{k_2}^* \Oo_{k_3}^*).$$
Thus, to conclude, it suffices to prove the isomorphism
 $$\Oo_{k_1}^\times\times\Oo_{k_2}^\times \times\Oo_{k_3}^\times  \;/\; \Oo_{k_1}^*\times \Oo_{k_2}^*\times \Oo_{k_3}^*\simeq  \Oo_{k_1}^\times\Oo_{k_2}^\times \Oo_{k_3}^\times \;/\; \Oo_{k_1}^*\Oo_{k_2}^*\Oo_{k_3}^*$$
since then $(\Oo_{k_1}^\times\times\Oo_{k_2}^\times \times\Oo_{k_3}^\times  : \Oo_{k_1}^*\times \Oo_{k_2}^*\times \Oo_{k_3}^*)=\prod_{i=1}^3 (\Oo_{k_i}^\times :\Oo_{k_i}^*)$. 
 
 \smallskip
 
 Let us consider the kernel of the surjective morphism
  $$\mu:  \Oo_{k_1}^\times\times\Oo_{k_2}^\times \times\Oo_{k_3}^\times \to  \Oo_{k_1}^\times\Oo_{k_2}^\times \Oo_{k_3}^\times \;/\; \Oo_{k_1}^*\Oo_{k_2}^*\Oo_{k_3}^*.$$ 
 $$\Ker(\mu)=\{(\eta_1,\eta_2,\eta_3)\in \Oo_{k_1}^\times\times\Oo_{k_2}^\times \times\Oo_{k_3}^\times \mid \eta_1\eta_2\eta_3\in \Oo_{k_1}^*\Oo_{k_2}^*\Oo_{k_3}^*\}.$$
In fact, the first part of the proof shows that, for every $i$, $\eta_i\in\Oo_{k_i}^*$, and hence, $(\eta_1,\eta_2,\eta_3)\in \Oo_{k_1}^*\times \Oo_{k_2}^*\times \Oo_{k_3}^*$.
\end{proof}

\begin{lemma}[Kubota's Hilfss\"atze 3 and 4]\label{HS3}
Assume that $K\subset\SR$ and $\lambda_1=\lambda_2=\lambda_3=-1.$ Let $\eta=\e_1\e_2\e_3$. Then, $\SQ(\eta)=K,$ the extension $\SQ(\sqrt{\eta})/\SQ$ is Galois and the exponent of its Galois group is $2$. Moreover, there exists $r\in\SQ\setminus\{0\}$ such that $K(\sqrt{\eta})=K(\sqrt{r})$ and, if we let $\rho=\sqrt{r\eta}$, then $\eta=\frac{\rho^2}{r}$ where $\rho\in K$ and $r\in\SQ$. 
\end{lemma}

\begin{proof}
The element $\eta$ of $K$ cannot belong to a quadratic subfield since, if for instance $\eta\in k_1$, then $\e_2\e_3\in k_1$. Whence, $\e_2\e_3=\sigma_1(\e_2\e_3)=\frac{-1}{\e_2}\frac{-1}{e_3},$ then $(\e_2\e_3)^2=1,$ which is impossible since by hypothesis $\e_i>1$. Finally, $K=\SQ(\eta)$. 

Thus $K\subseteq \SQ(\sqrt{\eta})$. If $\sqrt{\eta}\in K$ then, obviously, $\Gal(\SQ(\sqrt{\eta})/\SQ)=\Gal(K/\SQ)\simeq(\SZ/2\SZ)^2$ and we have $\eta=\frac{\rho^2}{r}$ with $\rho=\sqrt{\eta}$ and $r=1$. On the other hand, assume that $\sqrt{\eta}\not\in K$. Then the polynomial $X^2-\eta=0$ is irreducible over $K$ and $[\SQ(\sqrt{\eta}):\SQ]=8.$ The conjugates of $\eta$ with respect to $\SQ$, namely, $\eta$, $\sigma_1(\eta)=\frac{\e_1^2}{\eta}$, $\sigma_2(\eta)=\frac{\e_2^2}{\eta}$, $\sigma_3(\eta)=\frac{\e_3^2}{\eta},$ are the roots of the minimal polynomial $P(X)$ of $\eta$ over $\SQ$. The minimal polynomial of $\sqrt{\eta}$ over $\SQ$ is $P(X^2)$ whose roots, $\pm\sqrt{\eta}$, $\pm\frac{\e_1}{\sqrt{\eta}}$, $\pm\frac{\e_2}{\sqrt{\eta}}$, $\pm\frac{\e_3}{\sqrt{\eta}}$, are in $\SQ(\sqrt{\eta})$. The extension $\SQ(\sqrt{\eta})/\SQ$ is Galois. Moreover, the  $\SQ$-automorphisms of $\SQ(\sqrt{\eta})$ distinct from the identity, determined by these conjugates, are all of order 2. For instance, let us consider $\tau$ determined by  $\tau(\sqrt{\eta})=\frac{\e_1}{\sqrt{\eta}}$. We have $\tau(\eta)=\frac{e_1^2}{\eta}=\sigma_1(\e_1),$ thus $\tau_{\vert K}=\sigma_1$, then $\tau^2(\sqrt{\eta})=\tau\left(\frac{\e_1}{\sqrt{\eta}}\right)=\sigma_1(\e_1)\tau(\sqrt{\eta})^{-1}=\sqrt{\eta}.$ Finally,  $\Gal(\SQ(\sqrt{\eta})/\SQ)$ whose order is 8 and whose exponent is 2 is isomorphic $\left(\SZ/2\SZ\right)^3$. 

Let $G_1$ be the subgroup of order 2 of $\Gal(\SQ(\sqrt{\eta})/\SQ)$ generated by the $K$-isomorphism of $\SQ(\sqrt{\eta})$ characterized by $\sqrt{\eta}\mapsto -\sqrt{\eta}$ and let $G_2$ be a subgroup of order 4 such that $\Gal(\SQ(\sqrt{\eta})/\SQ)= G_1\times G_2$. The field fixed by $G_1$ is $K$ while the field fixed by $G_2$ is a quadratic subfield $\SQ(\sqrt{r})$. Since $G_1$ and $G_2$ generate $\Gal(\SQ(\sqrt{\eta})/\SQ)$, $K\cap \SQ(\sqrt{r})=\SQ$ and, in particular, $\sqrt{r}\notin K$. Consequently, $K(\sqrt{r})=K(\sqrt{\eta})$ and the $K$-non-trivial automorphism of $K(\sqrt{\eta})$ send $\sqrt{\eta}$ on $-\sqrt{\eta}$ as well as $\sqrt{r}$ on $-\sqrt{r}$ while  $\rho=\sqrt{r\eta}$ is fixed, and hence, is in $K$. Finally, $\eta=\frac{\rho^2}{r}$.
\end{proof}

\begin{lemma}[Kubota's Hilfssatz 4]\label{lem:7}
Let $\eta\in\Oo_K^\times$. One can write $\eta=\frac{\rho^2}{r}$ where $\rho\in\Oo_K$ and $r\in\SQ$ if and only if $\eta\in\Oo_K^*$. In particular, $\pm(\Oo_K^\times)^{(2)}\subseteq\Oo_K^*.$
\end{lemma}

\begin{proof}
Assume that $\eta=\frac{\rho^2}{r}$ where $\rho\in\Oo_K$ and $r\in\SQ$. Then, $N_{K/\SQ}(\eta)=1$ since $N_{K/\SQ}(\eta)=\left(\frac{N_{K/\SQ}(\rho)}{r}\right)^2>0.$ For all $i$, let $\eta_i=\frac{N_{K/k_i}(\rho)}{r}$. As $N_{K/k_i}(\eta)=\left(\frac{N_{K/k_i}(\rho)}{r}\right)^2=\eta_i^2$, we have $\eta_i^2\in\Oo_{k_i}$, and $\eta_i\in\Oo_{k_i}$. Moreover, $N_{k_i/\SQ}(\eta_i)=\frac{N_{K/\SQ}(\rho)}{r^2}=\pm 1$ (the sign does not depend on $i$). Finally,
$$\eta_1\eta_2\eta_3=\frac{\rho \sigma_2(\rho)}{r} \frac{\rho \sigma_3(\rho)}{r} \frac{\rho \sigma_1(\rho)}{r} =\frac{ {\rho}^2}{r} \frac{N_{K/\SQ}(\rho)}{r^2}=\pm \eta.$$
Thus, $\eta\in \Oo_{k_1}^\times\Oo_{k_2}^\times\Oo_{k-3}^\times$.  If $K\not\subset\SR$, for instance $k_1\subseteq \SR$, then $\Oo_{k_2}^\times=\Oo_{k_2}^*$ and $\Oo_{k_3}^\times=\Oo_{k_3}^*$, consequently $\eta_1\in \Oo_{k_1}^*$, and finally, $\eta\in \Oo_{k_1}^*\Oo_{k_2}^*\Oo_{k_3}^*=\Oo_K^*$. If $K\subset\SR$, since $\eta=\frac{\rho^2}{r}$ with $r\in\SQ$, we see that $\eta$ is either totally positive or totally negative, and hence, $\eta\in\Oo_K^*$. 

Conversely, assume that $\eta\in\Oo_K^*$. If $\eta\in \Oo_{k_1}^*\Oo_{k_2}^*\Oo_{k_3}^*$, then $\eta=\eta_1\eta_2\eta_3$ where $\eta_i\in\Oo_{k_i}^*$ and, by Lemma~\ref{lem:ein}, each $\eta_i$ can be written as $\frac{\rho_i^2}{r_i}$ where $\rho_i\in\Oo_{k_i}^\times $ and $r_i\in\SQ$. Then, $\eta=\frac{(\rho_1\rho_2\rho_3)^2}{r_1r_2r_3}$. It remains the case where $K\subset \SR$ and the norm of the three fundamental units is equal to $-1$. It is enough to prove that $\e_1\e_2\e_3$ may be written under the wanted form. This is what said Lemma~\ref{HS3}.
\end{proof}


\section{Proof of Kubota's Satz Vier}

By construction, the next lemma is obvious 

\begin{lemma}
$\mathrm{Im}\,\psi\simeq \Ima\,\psi_0/(\Pp_K\cap  \Ima\,\psi_0)$ and $\vert \Ima\,\psi\vert=\left(\,\Ima\,\psi_0:(\Pp_K\cap  \Ima\,\psi_0)\,\right).$
\end{lemma}

From $\Pp o(k_1)\times\Pp o(k_2)\times\Pp o(k_3)/\Ker\,\psi \simeq \textrm{Im}\,\psi\simeq\textrm{Im}\,\psi_0\,/\,(\Pp_K\cap  \Ima\,\psi_0),$
we deduce
\be\label{eq:5bis} \vert\Ker\,\psi\vert =\frac{\prod_{i=1}^3\vert\Pp o(k_i)\vert}{(\textrm{Im}\,\psi_0:\Pp_K\cap\Ima\,\psi_0)}\ee

To compute the index $(\textrm{Im}\,\psi_0:\Pp_K\cap\Ima\,\psi_0)$ in the denominator of Formula~(\ref{eq:5bis}), we consider a sequence of subgroups of  $\Ii_K^G$ contained in each others:
$$H_0=j_{\SQ}^K(\Pp_{\SQ})\subseteq H_1= \{\,\rho_0\Oo_K\mid \rho_0^2\in\SQ\,\}\subseteq H_2= \Pp_K\cap \Ima\;\psi_0\subseteq H_3=\Ima\,\psi_0\subseteq \Ii_K^G$$
 and we compute the indices  of these groups between themselves. For instance, to compute $(\textrm{Im}\,\psi_0:\Pp_K\cap\Ima\,\psi_0)=(H_3:H_2)$ we introduce $H_0=j_{\SQ}^K(\Pp_{\SQ}),$ so that

\centerline{$(H_3:H_2)=\frac{(H_3:H_0)}{(H_2:H_0)}.$}

\begin{lemma}\label{lem.321}Denoting by $s_K$ the number of primes which ramify in $K$, we have
\be \label{eq:5ter} (H_3:H_0)=\left(\,\mathrm{Im}\,\psi_0 :j_{\SQ}^K(\Pp_{\SQ})\,\right)=2^{s_K}.\ee
\end{lemma}

\begin{proof}
An ambiguous ideal of $k_i$ is of the form $\prod_{p\in\SP}\left(\Pi_{p^{f_{p,i}}}(k_i)\right)^{m_{p,i}}$. Since the $\Pi_p(k_i)$'s corresponding to primes $p$ that are inert or split are principal generated by a rational, we can forget them in the quotient by $j_{\SQ}^K(\Pp_{\SQ})$ and consider only $\prod_{p\vert d_K}\left(\Pi_{p}(k_i)\right)^{m_{p,i}}$. Note also that, if $p$ is ramified in $k_i$, then it is also ramified in at least another  $k_j$ where $i\not=j$, but since $\Pi_p(k_i)\Oo_K=\Pi_p(k_j)\Oo_K$, we may consider only one of these $k_i$, that we denote by $k_{i(p)}$. Finally, we are led to consider products of the form $\prod_{p\vert d_K}\left(\Pi_p(k_{i(p)})\right)^{m_{p}}$. Moreover, as $\left(\Pi_p(k_{i(p)})\right)^2\Oo_K=p\Oo_K$, we may assume that the exponents are 0 or 1, which shows that the number of distinct products is equal to $2^{s_K}$.
\end{proof}

\begin{lemma}\label{lem:5}
$H_2=\Pp_K\cap\mathrm{Im}\,\psi_0=\{\,\rho\Oo_K \mid \rho\in K \textrm{ such that } \rho^2\Oo_K \in j_{\SQ}^K(\Pp_{\SQ})\,\}.$\end{lemma}

\begin{proof}
The exponents of the groups $\Pp o(k_i)$ is 2 thus, for every ambiguous ideal $\aa_i$ of $k_i$, $\aa_i^2$ is principal generated by a rational. Consequently, if an ideal $\aa_1\aa_2\aa_3\Oo_K$ in $\textrm{Im}\,\psi_0$ is principal generated by $\rho\in K$, then its square is generated by a rational. Thus, $\Pp_K\cap\Ima\,\psi_0$ is contained in the right hand side of the equality to prove.

Conversely, let $\rho\in K$ be such that $\rho^2\Oo_K\in\Ima\,\psi_0$. Since $\rho^2\Oo_K$ is ambiguous, $\rho\Oo_K$ also. Assume first that $e_2(K/\SQ)\not=4$. For every prime $p$ that is ramified in $K$, let $k_{i(p)}$ be one of the two quadratic subfields where $p$ is ramified. Then,  $\Pi_{p^{f_p}}(K)=\Pi_p(k_{i(p)})\Oo_K$ since $p$ is not ramified in the extension $K/k_{i(p)}.$ Thus, 
$$\rho\Oo_K=\prod_{p\vert d_K}(\Pi_{p^{f_p}}(K))^{m_p}=\prod_{p\vert d_K}(\Pi_p(k_{i(p)})^{m_p}\Oo_K\in\Ima\,\psi_0.$$ 
Assume  now that $e_2(K/\SQ)=4$. Then, $\Pi_2(k_i)\Oo_K=\Pi_2(K)^2$
 and letting $m_2=2m+z$ where $z=0$ or 1, we may write
$$\rho\Oo_K=\prod_{p\vert d_K}(\Pi_{p^{f_p}}(K))^{m_p}=\Pi_2(K)^{2m+z}\prod_{p\vert d_K,\, p\not=2}(\Pi_{p^{f_p}}(K))^{m_p}$$
$$=\Pi_2(K)^z\times \Pi_2(k_{i(2)})^{m}\Oo_K\prod_{p\vert d_K, \,p\not=2}(\Pi_p(k_{i(p)})^{m_p}\Oo_K$$ 
If $z=0$, as in the previous case, $\rho\Oo_K\in \Ima\,\psi_0$. Let us prove by contradiction that $z=1$ cannot hold. By squaring both sides, we would get: $r\Oo_K=r' \times \left(\Pi_2(K)\right)^2$ where $r, r'\in\SQ$. Squaring again both sides, we obtain: $\left(\frac{r}{r'}\right)^2\Oo_K=\left(\Pi_2(K)\right)^4=2\Oo_K,$ which means that there exists $\nu\in\Oo_K^\times$ such that  $2\nu=q^2$ where $q\in\SQ$. But, $\nu \in\Oo_K^\times\cap\SQ=\{\pm 1\}$ implies $2=q^2$, which is impossible. Therefore we have the inverse inclusion, and hence the announced equality.
\end{proof}

\begin{lemma}[Kubota's Hilfssatz 6]
The map $$\pi:\rho\Oo_K\in\Pp_K\cap\mathrm{Im}\,\psi_1\mapsto \rho^2/r\pmod{\pm (\Oo_K^\times)^{(2)}}\in\Oo_K^\times/\pm (\Oo_K^\times)^{(2)}$$ where $r\in\SQ$ and $\rho^2/r\in\Oo_K^\times$ is a morphism whose image is $\Oo_K^*\,/\pm (\Oo_K^\times)^{(2)}$ and whose kernel is $\{\,\rho_0\Oo_K\mid \rho_0^2\in\SQ\,\}$. Thus,
$$\Pp_K\cap\mathrm{Im}\,\psi_0 \, /\, \{\,\rho_0\Oo_K\mid \rho_0^2\in\SQ\,\}\simeq \Oo_K^*\,/\pm (\Oo_K^\times)^{(2)},$$
and hence,
\be \label{eq:6t}(H_2:H_1)=\left(\Pp_K\cap\mathrm{Im}\,\psi_0 \, :\, \{\,\rho_0\Oo_K\mid \rho_0^2\in\SQ\,\}\right) = \left(\Oo_K^*\,:\pm (\Oo_K^\times)^{(2)}\right).\ee
\end{lemma}

\begin{proof}
$\pi$ is well defined since, by Lemma~\ref{lem:5}, for every $\rho\Oo_K\in \Pp_K\cap\mathrm{Im}\,\psi_0,$ there exists $r\in\SQ$ such that $r\Oo_K=\rho^2\Oo_K$, and hence, such that $\rho^2/r\in\Oo_K^\times$. That $\pi$ is a morphism of groups is straightforward. That the image of $\pi$ is exactly $\Oo_K^*\,/\pm (\Oo_K^\times)^{(2)}$ follows from Lemma~\ref{lem:7}. Finally, assume that $\rho\Oo_K\in\Ker\,\pi,$ then $\rho^2/r=\pm \eta^2$ where $r\in\SQ$ and $\eta\in \Oo_K^\times$ and $\rho_0=\rho/\eta$ satisfies $\rho_0^2\in\SQ$ and $\rho\Oo_K=\rho_0\Oo_K$. 
\end{proof}

\begin{lemma}
\be\label{eq:7v} (H_1:H_0)=\left(\,\{\,\rho_0\Oo_K\mid \rho_0^2\in\SQ\,\} \, :\, \,j_{\SQ}^K(\Pp_{\SQ})\,\right) = \left\{\begin{array}{l} 4 \textrm{ if } \sqrt{-1}\notin K\\ 2 \textrm{ if } \sqrt{-1}\in K.\end{array}\right.\ee 	
\end{lemma}

\begin{proof}
If $\rho_0\in K$ is such that $\rho_0^2\in \SQ$, then $\rho_0\in k_i=\SQ(\sqrt{d_i})$ for some $i$, and hence, either $\rho_0=q$ or $\rho_0=q\sqrt{d_i}$ for some $q\in\SQ$. Assume that $\rho_0\Oo_K=\rho_0'\Oo_K$. Then, for instance, $q\sqrt{d_i}=q'\eta$ for some $\eta\in\Oo_K^\times$.
	Then, $\eta^2\in\SQ\cap\Oo_K^\times=\{\pm 1\}$ and $\eta\in\{\pm 1,\pm i\}$. But $\eta\not=\pm 1$ since else $\sqrt{d_i}$ would be rational and $\eta\in\{\pm \sqrt{-1}\}$ implies $ \sqrt{-1}\in K$. The other cases lead to the same conclusion.
	
	 Finally, if $ \sqrt{-1}\notin K$, $H_1$ contains 4 classes mod $H_0$, those of $(1)$, $(\sqrt{d_1})$,  $(\sqrt{d_2})$, and $(\sqrt{d_3})$,  while, if $\sqrt{-1}\in K$, for instance $d_1=-1$ and the classes of (1) and $(\sqrt{-1})$ are equal as well as those of $(\sqrt{d_2})=(\sqrt{-d_2})$.
	\end{proof}

\begin{lemma}
$\left(\Oo_K^\times:(\Oo_K^\times)^{2}\right)=\left\{\begin{array}{l} 16 \textrm{ if } K\subset\SR \\ \;4 \textrm{ if } K\not\subset \SR\end{array}\right.$ and 
\be\label{eq:8j} \left(\Oo_K^\times:(\pm\Oo_K^\times)^{2}\right)=\left\{\begin{array}{l} 8 \textrm{ if } K\subset\SR \\ 4 \textrm{ if } \sqrt{-1}\in K \\ 2 \textrm{ else}.\end{array}\right.\ee
\end{lemma}

\begin{proof}
If $K\subset \SR$ then, by Dirichlet's unit theorem, $K$ has a fundamental system of 3 units $u_1,u_2,u_3$ such that any $\eta\in\Oo_K^\times$ may be uniquely written $\eta=\pm u_1^{m_1}u_2^{m_2}u_3^{m_3}$ where $(m_1,m_2,m_3)\in\SZ^3$, thus $\Oo_K^\times\simeq\SZ^3\times(\SZ/2\SZ)$. Any $\eta^2\in(\Oo_K^\times)^2$ may be uniquely written $\eta^2=u_1^{2m_1}u_2^{2m_2}u_3^{2m_3},$ thus $\vert\Oo_K^\times/(\Oo_K^\times)^{(2)}\vert=2^4$. 

If $K\not\subset\SR$, $K$ has a unique fundamental unit $u$ and $\Oo_K^\times\simeq \mu_K\times\SZ$ where $\mu_K$ denotes the group of roots of unity. As $\left(\mu_K:(\mu_K)^{(2)}\right)=2$,  $\vert\Oo_K^\times/(\Oo_K^\times)^{(2)}\vert=2^2$. 

Then, we obtain Formula~(\ref{eq:8j}) by noticing that $\pm(\Oo_K^\times)^{(2)}=(\Oo_K^\times)^{(2)}$ if and only if $\sqrt{-1}\in K$. 
\end{proof}

\begin{lemma}[Kubota's Hilfssatz 16]
\be\label{eq;9k} (H_3:H_2)=\left(\,\mathrm{Im}\,\psi_0\,:\,(\Pp_K\cap\mathrm{Im}\,\psi_0)\,\right)  
=\left\{\begin{array}{ll}2^{s_K-5}\,(\Oo_K^\times:\Oo_K^*) & \mathrm{if}\; K\subset \SR\\ 2^{s_K-3}\,(\Oo_K^\times:\Oo_K^*) & \mathrm{if}\;  K\not\subset \SR \end{array}\right.\ee 
	\end{lemma}

\begin{proof}
Thanks to Formulas~(\ref{eq:5ter}), (\ref{eq:6t}), (\ref{eq:7v}), and (\ref{eq:8j}), we may write:
$$(H_3:H_2)=\frac{(H_3:H_0)}{(H_2:H_1)(H_1:H_0)}= =\frac{2^{s_K}}{\left(\,\Oo_K^*\,:\, \pm(\Oo_K^\times)^{(2)}\right)(H_1:H_0)} $$
$$=\frac{2^{s_K}(\Oo_K^\times:\Oo_K^*)}{\left(\,\Oo_K^\times\,:\, \pm(\Oo_K^\times)^{(2)}\right)(H_1:H_0)} =2^{s_K}(\Oo_K^\times:\Oo_K^*)\times \left\{\begin{array}{ll}\frac{1}{8\times 4} & if K\subset \SR\\ \frac{1}{4\times 2} & if \sqrt{-1}\in K \\ \frac{1}{2\times 4} & else\end{array}\right.$$
\end{proof}

\begin{proof}[{\bf Proof of Kubota's Satz 4}]
By Formulas~(\ref{eq:I}), (\ref{eq:A}), (\ref{eq:5bis}), and (\ref{eq;9k}), we have:
$$\vert\Ker\,\psi\vert =\frac{\prod_{i=1}^3\vert\Pp o(k_i)\vert}{(\Ima\,\psi_0:\Pp_K\cap\Ima\,\psi_0)}= \frac{\sum_{i=1}^3 2^{s_i-1-\nu_i}}{2^{s_K-(5\,\textrm{or}\,3)}(\Oo_K^\times:\Oo_K^*)}$$
$$ = \frac{2^{2s_K+i_2-3-\nu_K}}{2^{s_K-(5\,\textrm{or}\,3)}q_K(\Oo_{k_1}^\times\Oo_{k_2}^\times\Oo_{k_3}^\times:\Oo_K^*)} = \frac{1}{q_K}\frac{2^{s_K+i_2-\nu_K+(2\,\textrm{or}\,0)}}{\prod_{i=1}^3 (\Oo_{k_i}^\times:\Oo_{k_i}^*)} \left(\times 2 \begin{array}{l}\mathrm{if}\; K\subset\SR \\ \mathrm{and}\; \nu_K=0\end{array}\right)$$
For $K\subset\SR$, $(\Oo_{k_i}^\times:\Oo_{k_i}^*)=2$ if $\nu_i=0$ and 1 else, thus 
$\prod_{i=1}^3\vert (\Oo_{k_i}^\times:\Oo_{k_i}^*)\vert=2^{3-\nu_K}.$
Consequently, $\vert \Ker\;\psi \vert = \frac{1}{2q_K}{2^{s_K+i_2}} (\times 2 \;if\;\nu_K=0)$.

\noindent For $K\not\subset\SR$, $\prod_{i=1}^3\vert (\Oo_{k_i}^\times:\Oo_{k_i}^*)\vert=2^{1-\nu_K}.$ Consequently, $\vert \Ker\;\psi \vert = \frac{1}{2q_K}{2^{s_K+i_2}}$
\end{proof}



\end{document}